\documentclass[en, amspaper, secnum, nobib]{./cls/mymcls}
\title{Compact Operator Semigroups Applied to Dynamical Systems}
\begin{document}

\begin{abstract}
  In this paper we develop a systematic theory of compact operator semigroups on locally convex vector spaces. In particular we prove new and generalized versions of the mean ergodic theorem and apply them to different notions of mean ergodicity appearing in topological dynamics.\medskip
\end{abstract}

\maketitle
\section{Introduction}

In \cite{Koehl1994} and \cite{Koehl1995} A. Köhler introduced the \emph{enveloping operator semigroup} of a power bounded operator $T \in \mathscr{L}(X)$ on a Banach space $X$ as the closure of the semigroup $\EuScript{S}\coloneqq \{(T')^n\mid n \in \N_0\}$ with respect to the operator topology of pointwise convergence induced by the weak* topology on $X'$.\\
If the operator $T$ is the Koopman operator of a topological dynamical system $(K;\varphi)$, i.e.{,} $Tf = f \circ \varphi$ for each $f \in \mathrm{C}(K)$, the relation between this enveloping operator semigroup and the classical \emph{Ellis semigroup} is an interesting issue (see \cite{Koehl1995} and \cite{Glas2007a}).\\
In addition, A. Romanov used the semigroup given by the convex closure of $\EuScript{S}$ to examine mean ergodicity of operators with respect to the weak* topology ({see} \cite{Roma2011}) and{---in the setting of topological dynamics---}with respect to the topology of pointwise convergence ({see} \cite{Roma2013}).\medskip\\
In this paper we generalize the {concepts from above} and give a new and systematic approach to compact operator semigroups on locally convex spaces. {We then discuss its applications to dynamical systems.}\\
It turns out that these semigroups are right topological semigroups (see \cite{BeJuMi1978} and \cite{BeJuMi1989} for an introduction to their theory) with respect to the topology of pointwise convergence. We use this as our starting point and introduce the {new concept of }abstract Köhler semigroups with their basic properties ({see} \cref{jdlgscenario} and \cref{relations}).
In the subsequent section we discuss examples appearing in topological dynamics: the \emph{Köhler semigroup}, the \emph{Ellis semigroup} and the \emph{Jacobs semigroup} (see \cref{relationstop}). {We give a new operator theoretic characterization of the Ellis semigroup (see \cref{sgrtopdyn}) and prove that a metric topological dynamical system is tame if and only its convex Köhler semigroup is a Fréchet--Urysohn space (see \cref{chartame}).}\\
Our main results are contained in the last two sections. In Section 4 we return to the abstract viewpoint and generalize {a mean ergodic theorem established by A. Romanov in \cite{Roma2011} for a single operator on a dual Banach space with the weak* topology to left amenable operator semigroups acting on locally convex spaces ({see} \cref{mainfirstchap}). The result is then applied to extend a mean ergodic theorem of M. Schreiber (\cite{Schr2013a}) to right amenable operator semigroups on barrelled locally convex spaces (see \cref{normweakly}).}\\
In the {final} section we study {unique ergodicity, norm mean ergodicity and weak* mean ergodicity} for {right} amenable semitopological semigroups acting on a compact space. This approach as well as the result that transitive weak* mean ergodic systems are already uniquely ergodic ({see} \cref{mainrelations}) seem to be new.\\
At the end of this section we discuss mean ergodicity of tame metric systems. We give a new proof for unique ergodicity of minimal tame metric dynamical systems generalizing {results} of Glasner (\cite{Glas2007a}), Huang (\cite{Huan2006}) as well as Kerr and Li (\cite{KeLi2007}) to the case of amenable semigroup actions ({see} \cref{uniquetame}). {Finally we extend results of Romanov in \cite{Roma2013} and characterize weak* mean ergodicity for tame amenable semigroup actions ({see} \cref{pointwisemain}).}\medskip\\
{In this paper all vector spaces are complex. }Moreover, all topological vector spaces and compact spaces are assumed to be Hausdorff. Given topological vector spaces $X$ and $Y$ we denote the continuous linear mappings from $X$ to $Y$ by $\mathscr{L}(X,Y)$ and set $\mathscr{L}(X) \coloneqq \mathscr{L}(X,X)$.\\
We recall that a special class of locally convex spaces arise{s} from dual pairs ({see} \cite{Scha1999}, Chapter IV). More generally, take two vector spaces $X$ and $Y$ and a bilinear mapping $\langle \,\cdot\, ,\, \cdot \, \rangle\colon X\times Y\longrightarrow \mathbbm{C}$ that separates $X$, i.e., for each $0 \neq x \in X$ there is $y \in Y$ with $\langle x,y\rangle \neq 0$, and call the pair $(X,Y)$ a \emph{left{-}separating pair}. The seminorms $\rho_{y}$ given by $\rho_y(x)\coloneqq |\langle x,y\rangle |$ for $x \in X$ and $y \in Y$ define a (Hausdorff) locally convex topology $\sigma(X,Y)$ on $X$.\\
The canonical duality between a Banach space $X$ and its dual $X'$ clearly defines left-separating pairs $(X,X')$ and $(X',X)$. Other examples are discussed in Section 3.\\

\section{Abstract Köhler Semigroups}
We start with a general definition of Köhler semigroups. Examples and applications are given in the sections below.
\begin{definition}
Let $X$ be a locally convex space, $\EuScript{S} \subset \mathscr{L}(X)$ a subsemigroup and equip $X^X$ with the product topology. The \emph{(abstract) Köhler semigroup} of $\EuScript{S}$ is the closure $\EuScript{K}(\EuScript{S})\coloneqq \overline{\EuScript{S}} \subset X^X$.
\end{definition}
If $X$ carries the weak topology induced by a left-separating pair $(X,Y)$ we denote the Köhler semigroup associated to a semigroup $\EuScript{S} \subset \mathscr{L}(X)$ by $\EuScript{K}(\EuScript{S};X,Y)$.\smallskip\\
{Recall that a semigroup $S$ equipped with a topology is called \emph{right topological} if the mapping
\begin{align*}
S \longrightarrow S, \quad s \mapsto st
\end{align*}
is continuous for each $t \in S$. It is \emph{left topological} if
\begin{align*}
S \longrightarrow S, \quad s \mapsto ts
\end{align*}
is continuous for each $t \in S$ and \emph{semitopological} if it is both, left topological and right topological (see Section 1.3 of \cite{BeJuMi1989}).\\
The \emph{topological center} $\Lambda(S)$ of a right topological semigroup $S$ is 
\begin{align*}
\Lambda(S) \coloneqq \{s \in S\mid t \mapsto st \text{ is continuous}\}.
\end{align*}
The following basic properties of the Köhler semigroup are readily verified.}
\begin{lemma}\label{lemmaprop}
For an operator semigroup $\EuScript{S} \subset \mathscr{L}(X)$ on a locally convex space $X$ the following assertions hold.
\begin{enumerate}[(i)]
\item $\EuScript{K}(\EuScript{S})$ is a right topological semigroup.
\item The topological center $\Lambda(\EuScript{K}(\EuScript{S}))$ contains $\EuScript{S}$.
\item $\EuScript{K}(\EuScript{S})$ is compact if and only if $\EuScript{S}x$ is relatively compact for each $x \in X$.
\item If $\EuScript{K}(\EuScript{S}) \subset \mathscr{L}(X)$, then $\EuScript{K}(\EuScript{S})$ is semitopological.
\item If $\EuScript{S}$ is abelian, then $\EuScript{K}(\EuScript{S})$ is semitopological if and only if it is abelian.
\end{enumerate}
\end{lemma}
For {operator semigroups on} certain classes of locally convex spaces one can say more and we recall some definitions from this theory. For an introduction we refer to \cite{Scha1999} and \cite{Jarc1981}.
\begin{definition}\label{repetitiontvs}
A locally convex vector space $X$ is called 
\begin{enumerate}[(i)]
\item \emph{barrelled} if every radial, convex, circled and closed set is a zero neighborhood. 
\item \emph{Montel} if it {is} barrelled and every bounded subset is relatively compact.
\item \emph{quasi{-}complete} if every closed bounded subset is complete (with respect to the uniformity defined by the zero neighborhoods of $X$).
\end{enumerate}
\end{definition}
Banach spaces and even Fréchet spaces are barrelled. Important examples of Montel spaces are the space of all holomorphic functions $\mathrm{H}(\Omega)$ on an open connected subset $\Omega \subset \C$ equipped with the topology of compact convergence and the space of smooth functions $\mathrm{C}^\infty(\Omega)$ on an open set $\Omega \subset \R^n$ equipped with the topology of compact convergence and, in all derivatives. We refer to Section 11.5 of \cite{Jarc1981} for more about these spaces.\\
Besides Banach and Fréchet spaces the dual space of a Banach space equipped with the weak* topology is quasi{-}complete.
\begin{proposition}\label{jdlgscenario}
Consider an operator semigroup $\EuScript{S}\subset \mathscr{L}(X)$ on a locally convex space $X$.
\begin{enumerate}[(i)]
\item The Köhler semigroup $\EuScript{K}(\EuScript{S})$ consists of continuous mappings on $X$ if one of the following conditions is fulfilled.
\begin{enumerate}[(a)]
\item For each pointwise bounded subset $M \subset \mathscr{L}(X)$ the pointwise closure $\overline{M} \subset X^X$ is contained in $\mathscr{L}(X)$. 
\item The space $X$ is barrelled.
\item $X$ carries the weak topology $\sigma(X,X')$ associated to a barrelled topology on $X$.
\item The semigroup $\EuScript{S}\subset \mathscr{L}(X)$ is equicontinuous.
\end{enumerate}
\item If $X$ is Montel \color{red} and $\EuScript{S}$ is bounded\color{black}, then $\EuScript{K}(\EuScript{S})$ is a compact semitopological subsemigroup of $\mathscr{L}(X)$.
\end{enumerate}
\end{proposition}
\begin{proof}
Clearly, condition (a) implies $\EuScript{K}(\EuScript{S}) \subset \mathscr{L}(X)$. 
We show that the spaces in (b) and (c) satisfy (a). For a barrelled space $X$ this is a consequence of the Banach--Steinhaus Theorem ({see} Theorem III.4.6 in \cite{Scha1999}).\\
Now consider a pointwise bounded set $M$ of $\sigma(X,X')${-}continuous operators and take $T_{\alpha} \in M$ for $\alpha \in A$ and $T \in X^X$ with $\lim_{\alpha} \langle T_{\alpha}x,x'\rangle = \langle Tx,x'\rangle$ for all $x \in X$ and $x' \in X'$. By passing to the (algebraic) adjoint operator $T^* \colon X' \longrightarrow X^*$, we obtain
\begin{align*}
T^*x'(x) = \lim_{\alpha} T_{\alpha}'x'(x)
\end{align*}
for $x' \in X'$ and $x \in X$.\\
For given $x' \in X'$ the net $(T_{\alpha}'x')_{\alpha \in A}$ is contained in the bounded set $M'x' = \{S'x'\mid S \in M\} \subset X'$ which is (since $X$ is barrelled) relatively $\sigma(X',X)$-compact (see the corollary to IV.1.6 in \cite{Scha1999}). Thus we find a subnet of $(T_{\alpha}x')_{\alpha \in A}$ converging in the $\sigma(X',X)${-}topology to an element $y' \in X'$ which yields $T^*x' = y' \in X'$. We obtain $T^*X' \subset X'$ which shows that $T$ is $\sigma(X,X')${-}continuous (IV.2.1 in \cite{Scha1999}) and so (ii) is proved.\\
In the situation of (d) we obtain $\EuScript{K}(\EuScript{S}) \subset \mathscr{L}(X)$ by Theorem III.4.3 of \cite{Scha1999}.\\
From the definition of a Montel space and part (i) we immediately obtain (ii).
\end{proof}
At the end of this section we prove some relations between different Köhler semigroups. We recall that a continuous mapping between two right topological semigroups is a \emph{homomorphism of right topological semigroups} if it is multiplicative. A surjective homomorphism is called \emph{epimorphism}.
\begin{proposition}\label{relations}
Let $\EuScript{S}_i \subset \mathscr{L}(X_i)$ be operator semigroups on locally convex spaces $X_i$ for $i=1,2$.
\begin{enumerate}[(i)]
\item Assume that $X_2$ is a subspace of $X_1$ and $\EuScript{S}_2 = \EuScript{S}_1|_{X_2}$. If $\EuScript{K}(\EuScript{S}_1)$ and $\EuScript{K}(\EuScript{S}_2)$ are compact, then
\begin{align*}
\EuScript{K}(\EuScript{S}_1) \longrightarrow \EuScript{K}(\EuScript{S}_2), \quad S \mapsto S|_{X_2}
\end{align*}
is an epimorphism of right topological semigroups.
\item Assume that $\Phi \in \mathscr{L}(X_1,X_2)$ is a surjective map and $\Psi\colon \EuScript{S}_1 \longrightarrow \EuScript{S}_2$ is an epimorphism of semigroups with $\Phi(Sx) = \Psi(S) \Phi(x)$ for all $S \in \EuScript{S}_1$ and $x \in X_1$. If $\EuScript{K}(\EuScript{S}_1)$ is compact, then
\begin{align*}
\hat{\Psi}\colon \EuScript{K}(\EuScript{S}_1) \longrightarrow \EuScript{K}(\EuScript{S}_2), \quad S \mapsto \hat{\Psi}(S)
\end{align*}
with $\hat{\Psi}(S)(\Phi(x))\coloneqq \Phi(Sx)$ for $S \in \EuScript{S}_1$ and $x \in X$ is an epimorphism of right topological semigroups with $\hat{\Psi}|_{\EuScript{S}_1} = \Psi$.
\end{enumerate}
\end{proposition}
\begin{proof}
We start with (i). The mapping given by 
\begin{align*}
\pi\colon \EuScript{K}(\EuScript{S}_1) \longrightarrow X_1^{X_2}, \quad S \mapsto S|_{X_2}
\end{align*}
is clearly continuous and by compactness of $\EuScript{K}({\EuScript{S}_1})$ closed. Compactness of $\EuScript{K}(\EuScript{S}_2)$ implies
\begin{align*}
\EuScript{K}(\EuScript{S}_2) = \overline{\EuScript{S}_2}^{X_2^{X_2}} =\overline{\EuScript{S}_2}^{X_1^{X_2}}
\end{align*}
and thus
\begin{align*}
\pi(\EuScript{K}(\EuScript{S}_1)) = \pi(\overline{\EuScript{S}_1}) = \overline{\pi(\EuScript{S}_1)}^{X_1^{X_2}} = \overline{\EuScript{S}_2}^{X_2^{X_2}} = \EuScript{K}(\EuScript{S}_2).
\end{align*}
Multiplicativity of the restriction map is trivial and (i) follows.\smallskip\\
We now prove (ii). To this end, we first show that
\begin{align*}
\hat{\Psi}\colon \EuScript{K}(\EuScript{S}_1) \longrightarrow X_2^{X_2}, \quad S \mapsto \hat{\Psi}(S)
\end{align*}
is well{-}defined. {Take $x, y\in X_1$ with $\Phi(x) = \Phi(y)$ and $S \in \EuScript{K}(\EuScript{S}_1)$ and take $S_{\alpha} \in \EuScript{S}_1$ for $\alpha \in A$ such that $\lim_{\alpha} S_{\alpha}z = Sz$ for each $z \in X_1$.} We then obtain
\begin{align*}
\Phi(Sx) &= \Phi\left(\lim_{\alpha} S_{\alpha}x\right) = \lim_{\alpha} \Phi(S_{\alpha}x) = \lim_{\alpha} \Psi(S_{\alpha}) \Phi(x) \\&=\lim_{\alpha} \Psi(S_{\alpha}) \Phi(y) = \Phi\left(\lim_{\alpha} S_{\alpha}y\right) = \Phi(Sy).
\end{align*}
This proves that $\hat{\Psi}$ is well-defined. Continuity of $\hat{\Psi}$ is obvious and by the same arguments as in (i) we obtain $\hat{\Psi}(\EuScript{K}(\EuScript{S}_1)) = \EuScript{K}(\EuScript{S}_2)$. To finish the proof, consider $S,T \in \EuScript{K}(\EuScript{S}_1)$.
We then have
\begin{align*}
\hat{\Psi}(ST)(\Phi(x)) = \Phi(STx) = \hat{\Psi}(S)(\Phi(Tx))= \hat{\Psi}(S)(\hat{\Psi}(T)(\Phi(x))
\end{align*}
for each $x \in X$ and thus $\hat{\Psi}(ST) = \hat{\Psi}(S)\hat{\Psi}(T)$.
\end{proof}
\section{Compact Operator Semigroups in Topological Dynamics}
We now apply the concepts and results {from} Section 2 {in} topological dynamics. 
\begin{definition}
{A \emph{topological dynamical system} is a pair $(K;S)$ of a compact space $K$ and a semitopological semigroup $S$ acting on $K$ such that the mapping
\begin{align*}
S \times K \longrightarrow K, \quad (s,x) \mapsto sx
\end{align*}
is continuous.}
\end{definition}
\color{red}Occasionally we write $\varphi_s$ for the continuous mapping $K\longrightarrow K, \, x \mapsto sx$ for $s \in S$.\color{black}
\begin{definition}\label{koopman}
Let $(K; S)$ be a topological dynamical system.
\begin{enumerate}[(i)]
\item To each $s \in S$ we associate the \emph{Koopman operator} $T_{s} \in \mathscr{L}(\mathrm{C}(K))$ by setting $T_{s}f(x) \coloneqq f(sx)$ for $s \in S$ and $x \in K$.
\item The \emph{Koopman semigroup} $\EuScript{T}_{{S}}$ associated to $(K;{S})$ is
\begin{align*}
\EuScript{T}_{S}\coloneqq \{T_s\mid s \in S\} \subset \mathscr{L}(\mathrm{C}(K)).
\end{align*}
\end{enumerate}
\end{definition} 
{For an introduction to the theory of dynamical systems we refer to \cite{Elli1969}, \cite{Brow1976} and \cite{Glas2008}. For an operator theoretic approach to this topic see \cite{EFHN2015}.}\\
We now consider the adjoint semigroup $\EuScript{T}_{S}' \coloneqq \{T_s'\mid s \in S\}  \subset \mathscr{L}(\mathrm{C}(K)')$ on the dual space $\mathrm{C}(K)'$ which we identify with the Banach lattice of regular Borel measures on $K$. We also identify the Banach lattice of discrete measures on $K$ with
\begin{align*}
\ell^1(K) \coloneqq \left\{(a_x)_{x \in K}\in \C^K\mmid \sum_{x \in K} |a_x| < \infty\right\}.
\end{align*}
Given an $S${-}invariant probability measure $\mu \in \mathrm{C}(K)'$ we identify the Banach lattice of measures absolutely continuous with respect to $\mu$ with $\mathrm{L}^1(K,\mu)$. We then obtain operator semigroups associated to the left{-}separating pairs $(\mathrm{C}(K)',\mathrm{C}(K))$, $(\ell^1(K),\mathrm{C}(K))$ and $(\mathrm{L}^1(K,\mu),\mathrm{C}(K))$.
\begin{lemma}\label{lemtopdynsem}
Consider a topological dynamical system $(K;S)$ and the corresponding Koopman semigroup $\EuScript{T}_S$. The following assertions are true.
\begin{enumerate}[(i)]
\item The semigroup $\EuScript{T}_{S}'$ has relatively $\sigma(\mathrm{C}(K)',\mathrm{C}(K))${-}compact convex orbits $\mathrm{co}\, \EuScript{T}_{S}'\mu$ for all $\mu \in \mathrm{C}(K)'$.
\item The space $\ell^1(K)$ is $\EuScript{T}_{S}'${-}invariant and the semigroup $\EuScript{T}_{S}'|_{\ell^1(K)}$ has relatively $\sigma(\ell^1(K),\mathrm{C}(K))${-}compact orbits $\EuScript{T}_{S}'(a_x)_{x \in K}$ for all\\
$(a_x)_{x \in K} \in \ell^1(K)$.
\item For each $S${-}invariant probability measure $\mu \in \mathrm{C}(K)'$ the space $\mathrm{L}^1(K,\mu)$ is $\EuScript{T}_{S}'${-}invariant and the semigroup $\EuScript{T}_{S}'|_{\mathrm{L}^1(K,\mu)}$ has relatively \\$\sigma(\mathrm{L}^1(K,\mu),\mathrm{C}(K))${-}compact convex orbits $\mathrm{co}\,\EuScript{T}_{S}'h$ for all $h \in \mathrm{L}^1(K,\mu)$.
\end{enumerate}
\end{lemma}
\begin{proof}
The first assertion holds since all operators in  $\EuScript{T}_{\EuScript{S}}$ are contractions and {the unit ball of $\mathrm{C}(K)'$ is weak* compact by the Banach--Alaoglu theorem (see Theorem III.4.3 of \cite{Scha1999}).}\smallskip\\
Since $T_s'\delta_x = \delta_{sx}$ for $s \in S$ and $x \in K$ we obtain that $\ell^1(K)$ is $\EuScript{T}_S'${-}invariant. Now take $(a_x)_{x \in K} \in \ell^1(K)$ and a net $(T_{s_{\alpha}})_{\alpha \in A}$. Since $K^K$ is compact for the product topology, we find a subnet $(s_{\beta})_{\beta \in B}$ of $(s_{\alpha})_{\alpha \in A}$ such that $(s_{\beta}x)_{\beta \in B}$ converges to some element $s(x) \in K$ for each $x \in K$.
Now consider $(b_x)_{x \in K} \in \ell^1(K)$ given by
\begin{align*}
b_x = \sum_{\substack{y \in K\\ s(y) = x}}a_y
\end{align*}
for $x \in K$. Let $\varepsilon > 0$, $f \in \mathrm{C}(K)$ and choose $x_1,...,x_N \in K$ with 
\begin{align*}
\sum_{x \notin \{x_1,...,x_N\}}|a_x| \cdot 2\|f\| \leq \varepsilon.
\end{align*}
We find $\beta_0 \in B$ with $|f(s_{\beta}x_j)-f(s(x_j))| \cdot \sum_{k=1}^N |a_{x_k}| \leq \varepsilon$ for all $\beta \geq \beta_0$ and $j \in \{1,...,N\}$. We thus obtain 
\begin{align*}
\left| \langle f, T_{s_{\beta}}'(a_x)_{x \in K}-(b_x)_{x \in K}\rangle \right| \leq \sum_{x \in K}|a_x| \cdot |f(s_{\beta}x)-f(s(x))| \leq 2 \varepsilon
\end{align*}
for all $\beta \geq \beta_0$. This shows $\lim_{\beta} T_{s_{\beta}}'(a_x)_{x \in K}= (b_x)_{x \in K}$.\smallskip\\
Finally consider an $S${-}invariant probability measure $\mu \in \mathrm{C}(K)'$. If a measure $\nu \in \mathrm{C}(K)'$ is absolutely continuous with respect to $\mu$, then so is $T_s'\nu$ for each $s \in S$. To see this, take a Borel measurable $\mu${-}null set $A$. Since $\mu$ is invariant, we obtain $\mu(s^{-1}(A)) = 0$ and by absolute continuity of $\nu$ we conclude $T_s'\nu(A) = \nu(s^{-1}(A)) = 0$.\\
To check compactness we note that the set 
\begin{align*}
\mathrm{co}\, \EuScript{T}_S'|_{\mathrm{L^1(K,\mu)}} = \mathrm{co}\,\{T_s'|_{\mathrm{L}^1(K,\mu)}\mid s \in S\}
\end{align*}
consists of bi{-}Markov operators on $\mathrm{L}^1(K,\mu)$ and thus is relatively compact with respect to the weak operator topology on $\mathrm{L}^1(K,\mu)$ ({see} Theorem 13.8 in \cite{EFHN2015}). Therefore $
\mathrm{co}\,\EuScript{T}_S'h$ is relatively $\sigma(\mathrm{L}^1(K,\mu),\mathrm{L}^\infty(K,\mu))${-}compact and in particular relatively $\sigma(\mathrm{L}^1(K,\mu),\mathrm{C}(K))${-}compact for each $h \in \mathrm{L}^1(K,\mu)$.
\end{proof}
{We now introduce various enveloping semigroups.}
\begin{definition}\label{sgrtopdyn}
To a topological dynamical system $(K;S)$ we associate the following semigroups.
\begin{enumerate}[(i)]
\item The \emph{Köhler semigroup} $\EuScript{K}(K;S)\coloneqq \EuScript{K}(\EuScript{T}_{S}';\mathrm{C}(K)',\mathrm{C}(K))$.
\item The \emph{convex Köhler semigroup} $\EuScript{K}_{\mathrm{c}}(K;S)\coloneqq \EuScript{K}(\mathrm{co} \, \EuScript{T}_{S}';\mathrm{C}(K)',\mathrm{C}(K))$.
\item The \emph{Ellis semigroup} $\EuScript{E}(K;S)\coloneqq \EuScript{K}(\EuScript{T}_{S}';\ell^1(K),\mathrm{C}(K))$.
\item For an $S${-}invariant probability measure $\mu \in \mathrm{C}(K)'$ the \emph{Jacobs semigroup} $\EuScript{J}(K,\mu;S)\coloneqq \EuScript{K}(\EuScript{T}_{S}';\mathrm{L}^1(K,\mu),\mathrm{C}(K))$.
\item For an $S${-}invariant probability measure $\mu \in \mathrm{C}(K)'$ the \emph{convex Jacobs semigroup} $\EuScript{J}_{\mathrm{c}}(K,\mu;S)\coloneqq \EuScript{K}(\mathrm{co}\,\EuScript{T}_{S}';\mathrm{L}^1(K,\mu),\mathrm{C}(K))$.
\end{enumerate}
\end{definition}
\begin{remark}
In the proof of \cref{lemtopdynsem} we have already seen that $\EuScript{T}_S'$ consists of bi-Markov operators and this set is relatively compact with respect to the weak operator topology ({see} Theorem 13.8 in \cite{EFHN2015}). Since the $\sigma(\mathrm{L}^1(K,\mu),\mathrm{C}(K))$-topology is coarser than the $\sigma(\mathrm{L}^1(K,\mu),\mathrm{L}^{\infty}(K,\mu))${-}topology, we immediately obtain
\begin{align*}
&\EuScript{J}(K,\mu;S)= \EuScript{K}(\EuScript{T}_{S}';\mathrm{L}^1(K,\mu),\mathrm{L}^{\infty}(K,\mu))\subset \mathscr{L}(\mathrm{L}^1(K,\mu)),\\
&\EuScript{J}_{\mathrm{c}}(K,\mu;S)= \EuScript{K}(\mathrm{co}\,\EuScript{T}_{S}';\mathrm{L}^1(K,\mu),\mathrm{L}^{\infty}(K,\mu))\subset \mathscr{L}(\mathrm{L}^1(K,\mu)).
\end{align*}
In particular, the Jacobs and the convex Jacobs semigroup{s} are semitopo\-logical.
\end{remark}
\begin{remark}
Consider the classical Ellis semigroup $\mathrm{E}(K;S)$ of a {topological dynamical system} $(K;S)$ given as the closure \color{red}$\overline{\{\varphi_s \mid s \in S\}}\subset K^K$\color{black} ({see} \cite{Elli1960}). If we identify $K$ with its homeo\-morphic copy $\{\delta_x\mid x \in K\} \subset\mathrm{C}(K)'$ (where $\mathrm{C}(K)'$ is equipped with the weak* topology), one readily checks that the mapping
\begin{align*}
\EuScript{E}(K;S) \longrightarrow \mathrm{E}(K;S),\quad R \mapsto R|_K
\end{align*}
is an isomorphism of right topological semigroups (this is a simple consequence of the fact, that the linear hull of Dirac measures is norm dense in $\ell^1(K)$). Thus our Ellis semigroup is isomorphic to the classical one, but is now obtained as an operator semigroup.
\end{remark}
\begin{remark}
\cref{relations} shows that taking subsystems and factors of a topological dynamical system $(K;S)$ induces epimorphisms of the corresponding Köhler and Ellis semigroups.
\end{remark}
The following result is an immediate consequence of \cref{relations} (i) and shows the relations between these semigroups.
\begin{proposition}\label{relationstop}
Let $(K;S)$ be a topological dynamical system and $\mu \in \mathrm{C}(K)'$ an $S${-}invariant probability measure. We then have the following diagram of epimorphisms of right topological semigroups where the arrows are given by the canonical restriction maps.
\[
\xymatrix{
& \EuScript{K}(K;S) \ar[ld] \ar[rd] \ar@{}[r]|-*[@]{\subset}& \EuScript{K}_{{\mathrm{c}}}(K;S)\ar[rd]& \\
\EuScript{E}(K;S)& & \EuScript{J}(\mu) \ar@{}[r]|-*[@]{\subset}& \EuScript{J}_{{\mathrm{c}}}(\mu)
}\]
\end{proposition}
It is natural to ask when these epimorphisms are isomorphisms. If the topological dynamical system $(K;S)$ is \emph{weakly almost periodic}, i.e., if $\EuScript{T}_Sf$ is relatively weakly compact for each $f \in \mathrm{C}(K)$, we obtain the following.
\begin{proposition}
{If $(K;S)$ is a weakly almost periodic system, then the canonical epimorphism $\EuScript{K}(K;S) \longrightarrow \EuScript{E}(K;S)$ is an isomorphism.\\
If moreover $\mu \in \mathrm{C}(K)'$ is a strictly positive invariant probability measure, then the canonical epimorphisms $\EuScript{K}(K;S) \longrightarrow \EuScript{J}(\mu)$ and $\EuScript{K}_\mathrm{c}(K;S) \longrightarrow \EuScript{J}_\mathrm{c}(\mu)$ are also isomorphisms.}
\end{proposition}
\begin{proof}
Since $(K;S)$ is weakly almost periodic, $\EuScript{K}(\EuScript{T}_S;\mathrm{C}(K),\mathrm{C}(K)')$ is a compact semitopological semigroup and one readily checks that
\begin{align*}
\EuScript{K}(\EuScript{T}_S;\mathrm{C}(K),\mathrm{C}(K)') \longrightarrow \EuScript{K}(\EuScript{T}_S';\mathrm{C}(K)',\mathrm{C}(K))= \EuScript{K}(K;S), \quad S \mapsto S'
\end{align*}
is an isomorphism if the order of multiplication in $\EuScript{K}(\EuScript{T}_S;\mathrm{C}(K),\mathrm{C}(K)')$ is reversed. Thus each operator {in} $\EuScript{K}(K;S)$ is weak* continuous. Since {$\ell^1(K)$ is weak* dense $\mathrm{C}(K)'$, we obtain injectivity of the restriction map $\EuScript{K}(K;S) \longrightarrow \EuScript{E}(K;S)$.\smallskip\\
If $\mu \in \mathrm{C}(K)'$ is a strictly positive invariant probability measure, then $\mathrm{L}^1(K,\mu)$ is also weak* dense in $\mathrm{C}(K)'$, and thus $\EuScript{K}(K;S) \longrightarrow \EuScript{J}(\mu)$ is injective.\\
Since, by Krein's theorem ({see Theorem 11.4 in \cite{Scha1999} or Theorem C.11 in \cite{EFHN2015}}), the sets $\mathrm{co}\,\EuScript{T}_Sf$ are also relatively weakly compact, a similar argument shows that the remaining map is an isomorphism.}
\end{proof}
The question when the epimorphism $\EuScript{K}(K;S) \longrightarrow \EuScript{E}(K;S)$ is an isomorphism was first posed by J. S. Pym (\cite{Pym1989}) and then answered by A. Köhler (\cite{Koehl1995}) for {metrizable topological dynamical systems satisfying the following condition which is weaker than weak almost periodicity.}
\begin{definition}
A metric topological dynamical system $(K;S)$ is \emph{tame} if for each $f \in \mathrm{C}(K)$ the orbit $\EuScript{T}_Sf$ is relatively sequentially compact with respect to the product topology of $\C^K$.
\end{definition}
Such systems have been studied in detail by E. Glasner and M. Megrelishvili (\cite{GlMe2006}, \cite{Glas2006a}, \cite{Glas2007a}, \cite{Glas2007b}, \cite{GlMe2012}, \cite{GlMe2013} and \cite{GlMe2015}) and later by W. Huang in \cite{Huan2006}, D. Kerr and H. Li in \cite{KeLi2007} as well as {by} A. Romanov in \cite{Roma2013}.\\
{We recall that a topological space $X$ is a \emph{Fréchet--Urysohn space} (see page 53 of \cite{Enge1989}) if each subset $A \subset X$ satisfies
\begin{align*}
\overline{A} = \left\{x \in X\mid \textrm{ there is a sequence } (x_n)_{n \in \N} \textrm{ in } A \textrm{ with }x = \lim_{n \rightarrow \infty} x_n\right\}.
\end{align*}
A compact space $K$ is called \emph{Rosenthal compact} if it can be continuously embedded into the space of Baire {$1$} functions $\mathrm{B}_1(X)$ on a Polish space $X$. By results of J. Bou{r}gain, D. H. Fremlin and M. Talagrand ({see} \cite{BoFrTa1978}) every Rosenthal compact space is a Fréchet--Urysohn space.} Moreover, closed subspaces and countable products of Rosenthal {compact spaces} are Rosenthal compact (see Section c{-}17 of \cite{HaNaVa2003} for more properties of {such} spaces). This leads to the following characterizations of tameness.
\begin{proposition}\label{chartame}
For a metric topological dynamical system $(K;S)$ the following are equivalent.
\begin{enumerate}[(i)]
\item The system $(K;S)$ is tame.
\item $\EuScript{K}_{\mathrm{c}}(K;S)$ is a Rosenthal compact space.
\item $\EuScript{K}(K;S)$ is a Rosenthal compact space.
\item $\EuScript{E}(K;S)$ is a Fréchet--Urysohn space.
\end{enumerate}
\end{proposition}
\begin{proof}
{As above we identify $K$ with its homeomorphic copy $\{\delta_x \mid x \in K\}$ in $\mathrm{C}(K)'$. In particular, $(R'f|_K)(x) = \langle f, R \delta_x\rangle$ for $R \in \EuScript{K}_{\mathrm{c}}(K;S)$, $f \in \mathrm{C}(K)$ and $x \in K$.}\smallskip\\
Assume that (i) holds. Then, for each $f \in \mathrm{C}(K)$ the set $\mathrm{co}\, \EuScript{T}_Sf$ is also relatively sequentially compact with respect to the product topology of $\C^X$ and its closure is contained in the space of Baire {$1$} functions $\mathrm{B}_1(K)$ ({see} Corollary 5G of \cite{BoFrTa1978}). Take a dense subset $\{f_n\mid n \in \N\} \subset \mathrm{C}(K)$ (which is possible, since $K$ is metric). {We will show that
\begin{align*}
\Phi\colon \EuScript{K}_{\mathrm{c}}(K;S) \longrightarrow \prod_{k \in \N} \overline{\mathrm{co}\,\EuScript{T}_Sf_k}^{\C^K}, \quad R \mapsto (R'f_k|_K)_{k \in \N}
\end{align*}
is a continuous embedding. Then, since countable products and closed subspaces of Rosenthal compact {spaces} are Rosenthal compact, this will finally prove that $\EuScript{K}_{\mathrm{c}}(K;S)$ is Rosenthal compact.}\\
We first check that $\Phi$ is well{-}defined. To this end, take a net $(R_{\alpha}')_{\alpha \in A} \in \mathrm{co}\,\EuScript{T}_S'$ converging to $R \in \EuScript{K}_{\mathrm{c}}(K;S)$. We then obtain{
\begin{align*}
(R'f|_K)(x) = \langle R'f,\delta_x\rangle = \lim_{\alpha} \langle R_{\alpha}f,\delta_x\rangle = \lim_{\alpha} (R_{\alpha}f)(x)
\end{align*}
for all $f \in \mathrm{C}(K)$ and $x \in K$.} Thus $\Phi$ is well{-}defined and it is clearly continuous. Given a net $(R_{\alpha}')_{\alpha \in A} \in \mathrm{co}\,\EuScript{T}_S'$ converging to $R \in \EuScript{K}_{\mathrm{c}}(K;S)$ the main theorem of \cite{Rose1977} implies{
\begin{align*}
\int_K (R'f_k)|_K \,\mathrm{d}\mu = \lim_{\alpha} \int_K R_{\alpha}f_k \, \mathrm{d}\mu = \lim_{\alpha} \langle R_{\alpha}f_k,\mu \rangle = \langle f_k,R\mu \rangle
\end{align*}
for each $k \in \N$ and $\mu \in \mathrm{C}(K)'$.} This shows that $\Phi$ is injective.\smallskip\\
Since closed subspaces of Rosenthal {compact spaces} are Rosenthal compact, we immediately obtain that (ii) implies (iii). Moreover, continuous images of compact Fréchet--Urysohn spaces are again Fréchet--Urysohn, so (iii) implies (iv) {by \cref{relationstop}}.\smallskip\\
Finally assume $\EuScript{E}(K;S)$ to be Fréchet--Urysohn. Take $f \in \mathrm{C}(K)$ and a net $(T_{s_\alpha}f)_{\alpha \in A}$ with $s_{\alpha} \in S$ for each $\alpha \in A$ converging pointwise to $g \in \C^K$. By passing to a subnet we may assume that $(s_{\alpha})_{\alpha \in A}$ converges pointwise to some $\psi \in \mathrm{E}(K;S)$ {such that} $g = f \circ \psi$. Since $\mathrm{E}(K;S)$ is a Fréchet--Urysohn space, we find a sequence $(s_n)_{n \in \N}$ in $S$ converging to $\psi$ whereby
\begin{align*}
g = f \circ \psi = \lim_{n \rightarrow \infty} f \circ \color{red}\varphi_{s_n}\color{black} \in \mathrm{B}_1(K).
\end{align*}
Thus $\overline{\,\EuScript{T}_Sf_k}^{\C^K}\subset \mathrm{B}_1(K)$ and Corollary 5G of \cite{BoFrTa1978} yields the claim.
\end{proof}
\begin{remark}
The equivalence of assertions (i) and (iv) is well{-}known. Our proof is based on the arguments of Glasner and Megrelishvili ({see} Theorem 3.2 in \cite{GlMe2006}), but extends the result to the convex Köhler semigroup showing thereby that the classes D2 and D3 of dynamical systems in \cite{Roma2013} are actually the same.
\end{remark}
{The next result is known (see Theorem 1.5 of \cite{Glas2006a} for the case of a group action), but for the sake of completeness we give a short proof.}
\begin{proposition}
Let $(K;S)$ be a tame metric topological dynamical system { and for $\psi \in \mathrm{E}(K;S)$ define the \emph{Koopman operator} $T_\psi \in \mathscr{L}(\mathrm{C}(K), \mathrm{B}_1(K))$ by $T_{\psi}f\coloneqq f \circ \psi$ for $\psi \in \mathrm{E}(K;S)$ and $f \in \mathrm{C}(K)$. Then
\begin{align*}
J\colon \mathrm{E}(K;S) \longrightarrow \EuScript{K}(K;S), \quad \psi \mapsto T_{\psi}'|_{\mathrm{C}(K)'}
\end{align*}
is an epimorphism of right topological semigroups and is the inverse} to the canonical restriction map
\begin{align*}
\EuScript{K}(K;S) \longrightarrow \mathrm{E}(K;S), \quad S \mapsto S|_K.
\end{align*}
\end{proposition}
\begin{proof}
A map from a sequential space to a Hausdorff space is continuous if and only if it is sequentially continuous ({see} Proposition 1.6.15 in \cite{Enge1989}). It is therefore a direct consequence of Lebesgue's Theorem that $J$ is continuous. Multiplicativity is trivial and, since $J(S) = \EuScript{T}_S'$, $J$ is surjective.
\end{proof}
\begin{remark}
Even for a tame system $(K;S)$ and a strictly positive invariant probability measure $\mu \in \mathrm{C}(K)'$ the epimorphism $\EuScript{K}(K;S) \longrightarrow \EuScript{J}(K,\mu;S)$ is not injective in general. In fact, if $S$ is abelian, then so is $\EuScript{J}(K,\mu;S)$ ({see} \cref{lemmaprop} (iv)), but $\EuScript{K}(K;S)$ generally not (see \cite{Glas2007a}, Example 4.5).
\end{remark}
In Section 5 we use the semigroups introduced above to study qualitative properties (e.g., mean ergodicity) of dynamical systems. 
\section{Mean Ergodic Semigroups}
Inspired by the approach of R. Nagel to mean ergodic semigroups ({see} \cite{Nage1973} and the supplement of Chapter 8 of \cite{EFHN2015}) we use techniques developed by A. Romanov in \cite{Roma2011} as well as M. Schreiber in \cite{Schr2013} and \cite{Schr2013a} to discuss mean ergodicity of operator semigroups on locally convex spaces (see also \cite{Eber1949}, \cite{Sato1977}, Section 2.1.2 in \cite{Kren1985}, \cite{GerKun2012}; see \cite{ABR2011} for mean ergodicity of one{-}parameter semigroups).  
\begin{definition}\label{defergodic}
Let $X$ be a locally convex space and $\EuScript{S} \subset \mathscr{L}(X)$ {be} an operator semigroup.
\begin{enumerate}[(i)]
\item A net $(T_{\alpha})_{\alpha \in A} \subset \overline{\mathrm{co}\, \EuScript{S}}^{X^X}$ is called
\begin{enumerate}[(a)]
\item a \emph{left ergodic net for} $\EuScript{S}$ if
\begin{align*}
\lim_{\alpha} (\mathrm{Id}-T)T_{\alpha}x = 0
\end{align*}
for each $x \in X$ and $T \in \EuScript{S}$.
\item a \emph{right ergodic net for} $\EuScript{S}$ if
\begin{align*}
\lim_{\alpha} T_{\alpha}(\mathrm{Id}-T)x = 0
\end{align*}
for each $x \in X$ and $T \in \EuScript{S}$.
\item a \emph{two{-}sided ergodic net for} $\EuScript{S}$ if it is left and right ergodic for $\EuScript{S}$.
\end{enumerate}
\item The semigroup $\EuScript{S}$ is called
\begin{enumerate}[(a)]
\item \emph{left mean ergodic} if each {left ergodic net} for $\EuScript{S}$ is pointwise convergent.
\item \emph{right mean ergodic} if each {right ergodic net} for $\EuScript{S}$ is pointwise convergent.
\end{enumerate}
\end{enumerate}
\end{definition}
We {recall some examples (see \cite{Schr2013a}, Examples 1.2).}
\begin{example}\label{ergodicoperatornets}
\begin{enumerate}[(i)]
\item Consider $\EuScript{S} = \{T^n\mid n \in \N_0\}$ for an operator $T \in \mathscr{L}(X)$ with bounded orbits $\EuScript{S}x$ for $x \in X$ on a locally convex space $X$. The sequence $(A_N)_{N \in \N}$ of \emph{Cesàro means} defined by
\begin{align*}
A_Nx \coloneqq \frac{1}{N}\sum_{n=0}^{N-1}T^nx
\end{align*}
for $x \in X$ and $N \in \N$ is a two{-}sided ergodic sequence for $\EuScript{S}$. 
\item For a pointwise bounded strongly continuous semigroup $(T(t))_{t \geq 0}$ (i.e., $t\mapsto T(t)x$ is continuous for each $x \in X$) on a quasi{-}complete space $X$ ({see} \cref{repetitiontvs} (iii)) we set 
\begin{align*}
A_sx \coloneqq \frac{1}{s}\int_0^sT(t)x\,\mathrm{d}t
\end{align*}
for $x \in X$ and $s \in (0, \infty)$ where the integral is understood in the sense of Bourbaki ({see} Proposition III.3.7 in \cite{Bour1965}). Then $(A_s)_{s > 0}$ is a two{-}sided {ergodic net for the semigroup $\{T(t)\mid t \geq 0\}$}.
\item The net of \emph{Abel means} $(S_r)_{r \in (1,\infty)}$ for an operator $T \in \mathscr{L}(X)$ with bounded orbits on a quasi{-}complete, barrelled locally convex space $X$ defined by
\begin{align*}
S_rx\coloneqq (r-1)\sum_{n=0}^{\infty}\frac{1}{r^{n+1}}T^nx
\end{align*}
for $x \in X$ and $r \in (1,\infty)$ is a two{-}sided {ergodic net} for $\EuScript{S} = \{T^n\mid n \in \N_0\}$.
\item {Example (ii) can be generalized as follows.}
{Consider a locally compact group $G$ with left Haar measure $\mu$ and let $S \subset G$ be a subsemigroup.
Assume further that there is a $\mu$\emph{{-}F{\o}lner net} $(F_{\alpha})_{\alpha \in A}$ in $S$, i.e., $F_{\alpha} \subset S$ is a compact set with positive finite measure for each $\alpha \in A$ and 
\begin{align*}
\lim_{\alpha} \frac{\mu(F_{\alpha} \triangle sF_{\alpha})}{\mu(F_{\alpha})} = 0
\end{align*} 
for each $s \in S$. Given a {pointwise bounded} representation
\begin{align*}
S \longrightarrow \mathscr{L}(X), \quad s \mapsto T(s)
\end{align*}
of $S$ on a quasi{-}complete {locally convex }space $X$ we can define 
\begin{align*}
\EuScript{F}_{\alpha}x\coloneqq \frac{1}{\mu(F_{\alpha})} \int_{F_{\alpha}} T(s)x\, \mathrm{d}\mu.
\end{align*}
for each $x \in X$ and each $\alpha \in A$ and thereby obtain a left {ergodic net} for {$\{T(s)\mid s \in S\}$}.}
\end{enumerate}
\end{example}
{Recall that a semitopological semigroup $S$ is \emph{left amenable} if the space $\mathrm{C}_{\mathrm{b}}(S)$ of bounded continuous functions on $S$ has a \emph{left invariant mean}, i.e., a positive element $m \in \mathrm{C}_{\mathrm{b}}(S)'$ with $m(\mathbbm{1}) = 1$ and $m(L_Sf) = m(f)$ for each $f \in \mathrm{C}_{\mathrm{b}}(S)$ and $s \in S$, where $L_sf(t) \defeq f(st)$ for every $t \in S$ (see Section 2.3 of \cite{BeJuMi1989}). \emph{Right amenability and (two-sided) amenability} are defined analogously.\\
We now assume that the semigroup $\EuScript{S} \subset \mathscr{L}(X)$ endowed with the topology of pointwise convergence is left amenable which} is always the case if $\EuScript{S}$ is abelian. In this situation we can characterize the convergence of all {left ergodic nets} through an algebraic property of the {Köhler} semigroup $\EuScript{K}(\mathrm{co}\, \EuScript{S})$.
\begin{theorem}\label{mainfirstchap}
For a locally convex space $X$ and a left amenable semigroup $\EuScript{S}\subset \mathscr{L}(X)$ with relatively compact convex orbits $\mathrm{co} \,\EuScript{S}x$ for all $x \in X$ the following are equivalent.
\begin{enumerate}[(i)]
\item The semigroup $\EuScript{S}$ is left mean ergodic.
\item The semigroup $\EuScript{K}(\mathrm{co}\,\EuScript{S})$ has a zero $Q$, i.e., $QS=SQ =Q$ for every $S \in \EuScript{K}(\mathrm{co}\,\EuScript{S})$.
\end{enumerate}
If these assertions hold, then $\lim T_{\alpha}x = Qx$ for all $x \in X$ and each left {ergodic net} $(T_{\alpha})_{\alpha \in A}$.
\end{theorem}
\cref{mainfirstchap} generalizes a result of A. Romanov (\cite{Roma2011}) for
\begin{align*}
\EuScript{S} = \{(T')^n\mid n \in \N_0\} \subset \mathscr{L}(X')
\end{align*}
and a dual Banach space $X'$ with the weak* topology. {For the proof we} use the methods developed by Romanov to cover our more general setting{, but need some lemmas.
In the first one} we describe the \emph{kernel} of $\EuScript{K}(\mathrm{co} \, \EuScript{S})$, i.e., the intersection of all ideals ({see} Notation 1.2.3 in \cite{BeJuMi1989}).
%The main ingredient for the proof of the theorem lies in the description of the kernel of the Köhler semigroup $\EuScript{K}(\mathrm{co}\,\EuScript{S})$. 
\begin{lemma}\label{kerneldesc}
Consider a locally convex space $X$ and a left amenable semigroup $\EuScript{S}\subset \mathscr{L}(X)$ with relatively compact convex orbits $\mathrm{co} \,\EuScript{S}x$ for $x \in X$. Then the kernel of $\EuScript{K}(\mathrm{co}\,\EuScript{S})$ is given by
\begin{align*}
\ker(\EuScript{K}(\mathrm{co}\, \EuScript{S})) &= \{Q \in \EuScript{K}(\mathrm{co}\,\EuScript{S})\mid \EuScript{S}Q = \{Q\}\}\\
&= \{Q \in \EuScript{K}(\mathrm{co}\,\EuScript{S})\mid Q \text{ is a right zero}\}.
\end{align*}
\end{lemma}
\begin{proof}
Set $I\coloneqq \{Q \in \EuScript{K}(\mathrm{co}\,\EuScript{S})\mid \EuScript{S}Q = \{Q\}\}$ and consider the continuous mappings
\begin{align*}
\lambda_T\colon \EuScript{K}(\mathrm{co}\,\EuScript{S}) \longrightarrow \EuScript{K}(\mathrm{co}\,\EuScript{S}), \quad S \mapsto TS
\end{align*}
for $T \in \EuScript{S}$. Since $\EuScript{S}$ is left amenable, these mappings have a common fixed point (see \cite{Day1961}, Theorem 3) which yields $I \neq \emptyset$. {If $Q \in I$, then clearly $\mathrm{co} \,\EuScript{S} = \{Q\}$ and therefore $\EuScript{K}(\mathrm{co}\,\EuScript{S})Q = \{Q\}$ since the semigroup is right topological. Thus, each element of $I$ is a right zero.} Moreover, $I$ is an ideal and we thus have $\ker(\EuScript{K}(\mathrm{co}\, \EuScript{S})) \subset I$. On the other hand, each right zero $Q$ of $\EuScript{K}(\mathrm{co}\, \EuScript{S})$ is a minimal idempotent and thus satisfies 
\begin{align*}
\{Q\} = \EuScript{K}(\mathrm{co}\, \EuScript{S})Q \subset \ker(\EuScript{K}(\mathrm{co}\, \EuScript{S})),
\end{align*}
by Theorem 1.2.12 of \cite{BeJuMi1989}.
\end{proof}
\begin{lemma}\label{charer}
Consider a locally convex space $X$ and a left amenable semigroup $\EuScript{S}\subset \mathscr{L}(X)$ with relatively compact convex orbits $\mathrm{co} \,\EuScript{S}x$ for $x \in X$. For a net $(T_{\alpha})_{\alpha \in A} \subset \EuScript{K}(\mathrm{co} \, \EuScript{S})$ the following are equivalent.
\begin{enumerate}[(i)]
\item {The net $(T_{\alpha})_{\alpha \in A}$ is left ergodic.}
\item All accumulation points of $\{T_{\alpha}\mid \alpha \in A\}$ for the operator topology of pointwise convergence are contained in $\ker(\EuScript{K}(\mathrm{co}\, \EuScript{S}))$.
\end{enumerate}
\end{lemma}
\begin{proof}
Assume that $(T_{\alpha})_{\alpha \in A}$ is left ergodic and take $T \in \EuScript{S}$. The map
\begin{align*}
\EuScript{K}(\mathrm{co}\,\EuScript{S}) \longrightarrow X^X, \quad S \mapsto (\mathrm{Id}-T)S
\end{align*}
is continuous. Thus, for each accumulation point $Q$ of $(T_{\alpha})_{\alpha \in A}$, the operator $(\mathrm{Id}-T)Q$ is an accumulation point of $((\mathrm{Id}-T)T_{\alpha})_{\alpha \in A}$. The assumption yields $(\mathrm{Id}-T)Q= 0$ and therefore $Q = TQ$. Since $T \in \EuScript{S}$ was arbitrary, we obtain $Q \in \ker(\EuScript{K}(\mathrm{co}\, \EuScript{S}))$ by \cref{kerneldesc}.\\
Assume now that there is $T \in \EuScript{S}$ such that $((\mathrm{Id}-T)T_{\alpha})_{\alpha \in A}$ does not converge to zero. We then find a zero neighborhood $U$ and a subnet $(T_{\beta})_{\beta \in B}$ of $(T_{\alpha})_{\alpha \in A}$ with $(\mathrm{Id}-T)T_{\beta} \notin U$ for all $\beta \in B$. By compactness of $\EuScript{K}(\mathrm{co}\, \EuScript{S})$ we find an accumulation point $Q$ of this subnet satisfying $(\mathrm{Id}-T)Q \neq 0$. Hence $TQ \neq Q$ and thus $(T_{\alpha})_{\alpha \in A}$ has an accumulation point which is not contained in $\ker(\EuScript{K}(\mathrm{co}\, \EuScript{S}))$.
\end{proof}
\begin{proof}[of \cref{mainfirstchap}]
By \cref{kerneldesc} the second assertion is equivalent to $\ker(\EuScript{K}(\mathrm{co}\,\EuScript{S}))$ being the singleton $\{Q\}$ and thus \cref{charer} shows the implication ``(ii) $\Rightarrow$ (i)''.\\
Now suppose that (i) is valid and take $Q_1,Q_2 \in \ker(\EuScript{K}(\mathrm{co}\,\EuScript{S}))$. {Take} a family of seminorms $P$ generating the topology on $X$. We define a partial order on the set 
\begin{align*}
A\coloneqq \{(Y,M,k)\mid Y \subset X \text{ finite}, M \subset P \text{ finite}, k \in \N\}
\end{align*}
by saying that $(Y,M,k) \leq (\tilde{Y},\tilde{M},\tilde{k})$ if $Y \subset \tilde{Y}$, $M \subset \tilde{M}$ and $k \leq \tilde{k}$. This order turns $A$ into a directed set. For each triple $\alpha = (Y,M,k) \in A$ we find $T_{1,\alpha},T_{2,\alpha} \in \mathrm{co} \, \EuScript{S}$ with 
\begin{align*}
\rho(T_{i,\alpha}x-Q_ix) \leq \frac{1}{k}
\end{align*}
for all $x \in Y$, $\rho \in M$ and $i=1,2$. The net given by 
\begin{align*}
T_{\alpha}\coloneqq \begin{cases} T_{1,\alpha} & \text{ if } \alpha = (Y,M,2n) \text{ for } n \in \N,\\
T_{{2},\alpha} & \text{ if } \alpha = (Y,M,2n-1) \text{ for } n \in \N,
\end{cases}
\end{align*}
is left ergodic by \cref{charer} and hence convergent. This yields $Q_1 = Q_2$. 
\end{proof}
We now introduce different notions of mean ergodicity on barrelled spaces. Given a pointwise bounded operator semigroup $\EuScript{S} \subset \mathscr{L}(X)$ on such a space, the Köhler semigroup $\EuScript{K}(\mathrm{co}\,\EuScript{S}';X',X)$ for the convex hull of the adjoint semigroup
\begin{align*}
\EuScript{S}' = \{S'\mid S \in \EuScript{S}\}
\end{align*}
is a compact right topological semigroup. To apply the results obtained above we assume that $\EuScript{S}$ is right amenable{, hence $\EuScript{S}'$ is left amenable. For a net $(T_\alpha)_{\alpha \in A}$ which is right ergodic for $\EuScript{S}$ with respect to the $\sigma(X,X')$-topology the adjoint net $(T_\alpha')_{\alpha \in A}$ is left ergodic for $\EuScript{S}'$ with respect to the $\sigma(X',X)$-topology. }\\
{T}he following definitions are natural.
\begin{definition}
Let $X$ be a barrelled space. A semigroup $\EuScript{S} \subset \mathscr{L}(X)$ is called 
\begin{enumerate}[(i)]
\item \emph{weak* mean ergodic} if $\EuScript{S}'$ is left mean ergodic with respect to the $\sigma(X',X)${-}topology.
\item \emph{weakly mean ergodic} if $\EuScript{S}$ is right mean ergodic with respect to the $\sigma(X,X')${-}topology.
\item \emph{strongly mean ergodic} if $\EuScript{S}$ is right mean ergodic with respect to the given topology on $X$.
\end{enumerate}
\end{definition}
Applying \cref{mainfirstchap} to the $\sigma(X',X)${-}topology immediately gives a characterization of weak* mean ergodicity.
{Next} we characterize weak and strong mean ergodicity (see also Theorem 1.7 {in} \cite{Nage1973}) extending results of {M.} Schreiber {for operator semigroups on Banach spaces} {to barrelled locally convex} spaces ({see} Theorem 1.7 in \cite{Schr2013a}, see also Corollary 1 of \cite{Sato1977} for a similar result). For a familiy $\EuScript{T}$ of operators on a locally convex space $X$ we use the notation 
\begin{align*}
&\fix(\EuScript{T}) \defeq \{x \in X\mid Tx = x \textrm{ for each } T \in \EuScript{T}\},\\
&\mathrm{rg}(\EuScript{T}) \defeq \{y \in X\mid \textrm{there are } x \in X \textrm{ and } T \in \EuScript{T} \textrm{ with } Tx = y\}.
\end{align*}
\begin{theorem}\label{normweakly}
Consider a bounded right amenable semigroup $\EuScript{S} \subset \mathscr{L}(X)$ on a barrelled \color{red} quasi-complete \color{black} locally convex space {$(X,\tau)$}. Then the following assertions are equivalent.
\begin{enumerate}[(i)]
\item There is a two{-}sided {ergodic net} $(T_{\alpha})_{\alpha \in A}$ for $\EuScript{S}$ with respect to the weak topology such that $(T_{\alpha}x)_{\alpha \in A}$ converges weakly for each $x \in X$.
\item The semigroup $\EuScript{S}$ is weakly mean ergodic.
\item The semigroup $\EuScript{S}$ is strongly mean ergodic.
\item The semigroup $\EuScript{K}(\mathrm{co}\,\EuScript{S};X,X')$ has a zero $P$.
\item The semigroup $\EuScript{K}(\mathrm{co}\,\EuScript{S}';X',X)$ has a zero $Q$ which is weak* continuous.
\item The fixed space $\fix(\EuScript{S})$ separates $\fix(\EuScript{S}')$.
\item $X = \fix(\EuScript{S}) \oplus \overline{\mathrm{lin}} \rg(\mathrm{Id}- \EuScript{S})$.
\end{enumerate}
If one of the above {assertions} is valid, then $\lim_{\alpha} T_{\alpha}x = Px$ in weak (resp. $\tau$) topology for each operator net $(T_{\alpha})_{\alpha \in A}$ which is {right ergodic} for $\EuScript{S}$ with respect to the weak (resp. $\tau$) topology.
\end{theorem}
{We start with the following lemma.
\begin{lemma}\label{prep}
Consider a bounded right amenable semigroup $\EuScript{S} \subset \mathscr{L}(X)$ on a  \color{red} quasi-complete \color{black} barrelled locally convex space $(X,\tau)$. Then the following assertions hold.
\begin{enumerate}[(i)]
\item There exist right ergodic nets for $\EuScript{S}$ with respect to the topology $\tau$. 
\item Let $D$ be the set of all $x \in X$ for which $\lim_\alpha T_\alpha x$ exists for every net $(T_\alpha)_{\alpha \in A}$ which is right ergodic with respect to $\tau$. Then $D$ is closed.
\item Let $D_0$ be the set of all $x \in X$ for which $\lim_\alpha T_\alpha x=0$ for every net $(T_\alpha)_{\alpha \in A}$ which is right ergodic with respect to $\tau$. Then $D_0$ is closed.
\item $\fix(\EuScript{S}) \cap \overline{\mathrm{lin}} \rg(\mathrm{Id}- \EuScript{S}) = \{0\}$ and $\fix(\EuScript{S}) \oplus \overline{\mathrm{lin}} \rg(\mathrm{Id}- \EuScript{S}) \subset D$.
\end{enumerate}
\end{lemma}
\begin{proof}
We first observe that there are right {ergodic nets} for $\EuScript{S}$ with respect to the $\sigma(X,X')$-topology. In fact, by \cref{charer} we find {a left ergodic net} $(T_\alpha')_{\alpha \in A}$ for $\EuScript{S}' \subset \mathscr{L}(X')$ with respect to the $\sigma(X',X)$-topology such that $T_{\alpha}' \in \mathrm{co}\,\EuScript{S}'$ for all $\alpha \in A$. The net $(T_{\alpha})_{\alpha \in A}$ of pre{-}adjoints is then {right} ergodic for $\EuScript{S}$ with respect to the {$\sigma(X,X')$-topology}.\\
{The proof of Theorem 1.4 in \cite{Schr2013a} (which still works in the case of locally convex spaces) now shows that there are actually even {right ergodic nets} for $\EuScript{S}$ with respect to the topology $\tau$.}\smallskip\\
Now if $(T_\alpha)_{\alpha \in A}$ is a right ergodic net for $\EuScript{S}$ with respect to the topology $\tau$, then the set $\{T_\alpha\mid \alpha \in A\} \subset \mathscr{L}(X)$ is equicontinuous since $(X,\tau)$ is barrelled (see Theorem III.4.2 of \cite{Scha1999}). Therefore $\lim_\alpha T_\alpha x$ exists for each $x \in \overline{D}$ by Theorem III.4.5 of \cite{Scha1999} and $D$ is closed. Similarly we see that $D_0$ is closed.\smallskip\\
Since $\overline{\mathrm{lin}} \rg(\mathrm{Id}- \EuScript{S}) \subset D_0$ by (ii), we have $\fix(\EuScript{S}) \cap \overline{\mathrm{lin}} \rg(\mathrm{Id}- \EuScript{S}) = \{0\}$. Moreover, we obtain $\fix(\EuScript{S}) \oplus \overline{\mathrm{lin}} \rg(\mathrm{Id}- \EuScript{S}) \subset D$. 
\end{proof}
}
\begin{proof}[of \cref{normweakly}]
We {first} prove that (ii) implies (i). So assume (ii) and take any right {ergodic net} $(T_{\alpha})_{\alpha \in A}$ for $\EuScript{S}$ {with respect to the $\sigma(X,X')$-topology}. {For each} $S \in \EuScript{S}$ the net $(ST_{\alpha})_{\alpha \in A}$ is also right ergodic {with respect to the $\sigma(X,X')$-topology}. Since all right {ergodic nets} converge, all of them {must} have the same limit (otherwise the ``mixed nets'' would not be convergent). Thus
\begin{align*}
\lim_{\alpha} (\mathrm{Id}-S)T_{\alpha}x = \lim_{\alpha} T_{\alpha}x-\lim_{\alpha} ST_{\alpha}x =0
\end{align*}
weakly for each $x \in X$ and hence $(T_{\alpha})_{\alpha \in A}$ is {also} left ergodic {with respect to the $\sigma(X,X')$-topology}.
A similar argument shows that (iii) implies (i). \smallskip\\
Let now $(T_{\alpha})_{\alpha \in A}$ be a {net} as in (i). We set $Px\coloneqq \lim_{\alpha} T_{\alpha}x$ for $x \in X$ where convergence is understood with respect to the weak topology $\sigma(X,X')$. Then $P \in \EuScript{K}(\mathrm{co}\,\EuScript{S};X,X')$ and $P$ is continuous with respect to the weak topology by \cref{jdlgscenario}. Moreover{,} we obtain
\begin{align*}
0 = \lim_{\alpha}T_{\alpha}(\mathrm{Id}-T)x= Px-PTx\\
0 = \lim_{\alpha}(\mathrm{Id}-T)T_{\alpha}x = Px - TPx
\end{align*}
for each $x \in X$ and $T \in \EuScript{S}$. This shows $PT =TP = P$ for each $T \in \EuScript{S}$ and consequently, since multiplication is separately continuous with respect to the weak operator topology, 
\begin{align*}
PR = P =RP
\end{align*}
for all $R \in \EuScript{K}(\mathrm{co}\,\EuScript{S};X,X')$ and thus (iv) holds.\smallskip\\
Suppose that (iv) is valid and let $P \in \EuScript{K}(\mathrm{co}\,\EuScript{S};X,X')$ be the zero element. We then obtain $P' \in \EuScript{K}(\mathrm{co}\,\EuScript{S}{'};X',X)$ and even $P' \in \ker \EuScript{K}(\mathrm{co}\,\EuScript{S}{'};X',X)$ by \cref{kerneldesc}. Now take any $Q \in \ker \EuScript{K}(\mathrm{co}\,\EuScript{S}';X',X)$ and a left {ergodic net} $(T_{\alpha}')_{\alpha \in A}\subset \mathrm{co}\,\EuScript{S}'$ for $\EuScript{S}'$ {with respect to the weak* topology such that $\lim_{\alpha} T_{\alpha}' = Q$.} Since each operator $T_{\alpha}'$ has a pre{-}adjoint in $\EuScript{K}(\mathrm{co}\,\EuScript{S};X,X')$ we obtain 
\begin{align*}
Q = P'Q = \lim_{\alpha} P'T_{\alpha}' = P'
\end{align*}
and hence $P' = Q$. Consequently, the kernel of $\ker \EuScript{K}(\mathrm{co}\,\EuScript{S}';X',X)$ consists only of $P'$ which shows that $P'$ is a weak* continuous zero.\smallskip\\
Now assume that (v) is satisfied. Let $Q= P'\in \EuScript{K}(\mathrm{co}\,\EuScript{S};X';X)$ be the weak* continuous zero and take $0 \neq x' \in \fix(\EuScript{S}')$. We find $x \in X$ with $\langle x,x'\rangle \neq 0$ and $y\coloneqq Px \in \fix(\EuScript{S})$ then satisfies $\langle y,x'\rangle = \langle x,Qx'\rangle = \langle x,x'\rangle \neq 0$. {Hence we have (vi).}\smallskip\\
{Suppose that (vi) holds. Take $x' \in X'$ vanishing on $\fix(\EuScript{S}) \oplus \overline{\mathrm{lin}} \rg(\mathrm{Id}- \EuScript{S})$. In particular $\langle x-Sx,x'\rangle = 0$ for all $x \in X$ and $S \in \EuScript{S}$ and hence $x' \in \fix(\EuScript{S}')$ and $x'=0$ since $\fix(\EuScript{S})$ separates $\fix(\EuScript{S}')$ and $x'$ vanishes on $\fix(\EuScript{S})$. Thus $\fix(\EuScript{S}) \oplus \overline{\mathrm{lin}} \rg(\mathrm{Id}- \EuScript{S})$ is dense in $X$ by the Hahn{--}Banach Theorem, and, by \cref{prep} (ii) $D=X$. Thus (vi) implies (iii).}\smallskip\\
\cref{mainfirstchap} shows that (v) implies (ii) and therefore the equivalence of assertions (i) -- (vi). The statement about the limit also follows from \cref{mainfirstchap}.\smallskip\\
The implication ``(vii) $\Rightarrow$ (iii)'' is clear. Conversely, if $(T_{\alpha})_{\alpha \in A}$ is a net as in (i), then $Px = \lim_\alpha T_\alpha x\in \fix(\EuScript{S})$ and 
\begin{align*}
x - Px = \lim_{\alpha} (\mathrm{Id}-T_{\alpha})x \in \overline{\mathrm{lin}} \rg(\mathrm{Id}- \EuScript{S}),
\end{align*}
which establishes (vii).
\end{proof}
\begin{corollary}\label{compactorbitsmean}
Every amenable operator semigroup $\EuScript{S} \subset \mathscr{L}(X)$ on a barrelled  \color{red} quasi-complete \color{black} locally convex space $X$ with relatively weakly compact convex orbits is strongly mean ergodic.
\end{corollary}
\begin{proof}
By compactness of $\EuScript{K}(\mathrm{co}\,\EuScript{S};X,X')$ the mapping
\begin{align*}
\Phi \colon \EuScript{K}(\mathrm{co}\,\EuScript{S};X,X')\longrightarrow \EuScript{K}(\mathrm{co}\,\EuScript{S}';X',X),\quad S \mapsto S'
\end{align*}
is an isomorphism of right topological semigroups if we reverse the order of multiplication in $\EuScript{K}(\mathrm{co}\,\EuScript{S};X,X')$. In particular, we obtain
\begin{align*}
\Phi(\mathrm{ker}(\EuScript{K}(\mathrm{co}\,\EuScript{S};X,X'))) = \mathrm{ker}(\EuScript{K}(\mathrm{co}\,\EuScript{S}';X',X)).
\end{align*}
Take $P \in \mathrm{ker}(\EuScript{K}(\mathrm{co}\,\EuScript{S};X,X'))$. Then $\EuScript{S}P = \{P\}$ and $\EuScript{S}'P' = \{P'\}$ by \cref{kerneldesc} and thus $SP=PS =P$ for each $S \in \EuScript{S}$. Since $\EuScript{K}(\mathrm{co}\,\EuScript{S};X,X')$ is semitopological, {we obtain $SP=PS =P$ for each $S \in \EuScript{K}(\mathrm{co}\,\EuScript{S};X,X')$, i.e., $P$ is a zero in $\EuScript{K}(\mathrm{co}\,\EuScript{S};X,X')$.}
\end{proof}
\begin{remark}
If $X$ is a reflexive barrelled space, then every bounded {amenable} semigroup is strongly mean ergodic by \cref{compactorbitsmean} ({see} Theorem IV.5.6 in \cite{Scha1999}).
\end{remark}
We present an example where \cref{normweakly} is applicable.
\begin{example}
Consider the space $\mathrm{C}(\R)$ of continuous functions on $\R$ equipped with the compact{-}open topology, i.e., the locally convex topology induced by the seminorms $\rho_K$ for $K \subset \R$ compact defined by
\begin{align*}
\rho_K(f) \coloneqq \sup_{x \in K} |f(x)|
\end{align*}
for all $f \in \mathrm{C}(\R)$. {Then $\mathrm{C}(\R)$ is a Fréchet space (and therefore barrelled)} and its dual space can be identified with the compactly supported Borel measures on $\R$ ({see} Corollary 7.{6.}5 in \cite{Jarc1981}).\\
Consider the multiplication operator $T \in \mathscr{L}(\mathrm{C}(\R))$ defined by 
\begin{align*}
(Tf)(x)\coloneqq |\cos(x)| \cdot f(x)
\end{align*}
for each $f \in \mathrm{C}(\R)$ and each $x \in \R$. Then $\EuScript{S}\coloneqq \{T^n\mid n \in \N_0\}$ is bounded. Moreover we have $\fix(\EuScript{S}) = \{0\}$ and $\lin \{\delta_{\pi k}\mid k \in \Z\} \subset \fix(\EuScript{S}')$. Thus $\EuScript{S}$ is not strongly mean ergodic by \cref{normweakly}. However, it is weak* mean ergodic. In fact, for each $f \in \mathrm{C}(\R)$ we obtain $\lim_{n \rightarrow \infty} T^nf = Pf$ pointwise with
\begin{align*}
(Pf)(x)\coloneqq \begin{cases}f(\pi k)& \text{if } x = \pi k \text{ with } k \in \Z, \\
0& \text{else}.
\end{cases}
\end{align*}
Lebesgue's Theorem implies
\begin{align*}
\lim_{n \rightarrow \infty} \langle f,(T')^n\mu \rangle = \int_{\R} Pf\,\mathrm{d}\mu = \langle f, \sum_{k \in \Z} \mu(\{\pi k\}) \delta_{\pi k}\rangle
\end{align*}
for each $f \in \mathrm{C}(\R)$ and each $\mu \in \mathrm{C}(\R)'$. Thus $\lim_{n \rightarrow \infty} (T')^n\mu = \sum_{k \in \Z} \mu(\pi k) \delta_{\pi k}$ in weak* topology which implies
\begin{align*}
\lim_{\alpha} S_{\alpha}\mu = \sum_{k \in \Z} \mu(\{\pi k\}) \delta_{\pi k}
\end{align*}
in weak* topology for each $\mu \in \mathrm{C}(\R)'$ and each left {ergodic net} $(S_{\alpha})_{\alpha \in A}$ for $\EuScript{S}'$.
\end{example}
\section{Mean Ergodicity in Topological Dynamics}
In this section we study different notions of mean ergodicity in topological dynamics (see \cite{Schr2014}). We note that for a topological dynamical system $(K;S)$ the mapping
\begin{align*}
S \longrightarrow \EuScript{T}_S, \quad s \mapsto T_s
\end{align*}
is an epimorphism of {semitopological semigroups, if we reverse the order of multiplication in $S$ and equip $\EuScript{T}_S$ with the strong operator topology ({see} Theorem 4.17 in \cite{EFHN2015}). In particular, if $S$ is left amenable, then $\EuScript{T}_S$ is right amenable.}\\
A topological dynamical system $(K;S)$ is said to be \emph{weak*} (resp. \emph{norm}) \emph{mean ergodic} if the Koopman semigroup $\EuScript{T}_S \subset \mathscr{L}(\mathrm{C}(K))$ is weak* (resp. strongly) mean ergodic. Moreover, the system $(K;S)$ is \emph{uniquely ergodic} if there is a unique $S${-}invariant probability measure $\mu \in \mathrm{C}(K)'$.
% and note that if $S$ is left amenable there always exist $S$--invariant probability measures on $K$.\\
{Using the convex Köhler semigroup ({see} \cref{sgrtopdyn}) we obtain the following characterization of unique ergodicity.}
\begin{proposition}\label{uniquee}
Let $(K;S)$ be a topological dynamical system with $S$ left amenable. The following assertions are equivalent.
\begin{enumerate}[(i)]
\item There is {a net} $(T_{\alpha})_{\alpha \in A} \subset \mathrm{co}\,\EuScript{T}_S$ for $\EuScript{T}_{S}$ which is right ergodic with respect to the weak topology such that for each $f \in \mathrm{C}(K)$ the net $(T_\alpha f)_{\alpha \in A}$ converges weakly to a constant function.
\item The system $(K;S)$ is uniquely ergodic.
\item The semigroup $\EuScript{K}_{\mathrm{c}}(K;S)$ has a zero which is a rank one operator.
\end{enumerate}
If one of these assertions holds and $\mu \in \mathrm{C}(K)'$ is the unique invariant probability measure, then
\begin{align*}
\lim_{\alpha} T_{\alpha} f = \int_K f\,\mathrm{d}\mu\cdot \mathbbm{1}
\end{align*}
uniformly on $K$ for each $f \in \mathrm{C}(K)$ and for {each net} $(T_{\alpha})_{\alpha \in A}$ which is right ergodic for $\EuScript{T}_S$ with respect to the norm topology.
\end{proposition}
\begin{proof}
Assume that $(T_\alpha)_{\alpha \in A}$ is {a net} as in (i). For each $f \in \mathrm{C}(K)$ {let} $c(f) \in \C$ with $\lim_{n \rightarrow \infty} T_\alpha f = c(f) \cdot \mathbbm{1}$. Then
\begin{align*}
\lim_{\alpha} (\mathrm{Id}-T_s) T_\alpha f = c(f) \cdot \mathbbm{1} - c(f)\cdot T_s\mathbbm{1} = 0
\end{align*}
weakly for each $f \in \mathrm{C}(K)$ and therefore $(T_\alpha)_{\alpha \in A}$ is also left ergodic. Thus assertion (i) of \cref{normweakly} holds. By (v) of \cref{normweakly} we obtain that $\fix(\EuScript{T}_{S})$ separates $\fix(\EuScript{T}_{S}')$. But for each $f \in \fix(\EuScript{T}_{S})$ we have 
\begin{align*}
f = \lim_{n \rightarrow \infty} T_\alpha f = c(f) \cdot \mathbbm{1},
\end{align*}
hence $\fix(\EuScript{T}_{S})$ is one dimensional and so must $\fix(\EuScript{T}_{S}')$ {proving (ii)}.\smallskip\\
Suppose that (ii) is valid and let $Q_1,Q_2 \in \ker(\EuScript{K}_{\mathrm{c}}(K;S))$. For each probability meausure $\mu \in \mathrm{C}(K)'$ the measures $Q_1\mu, Q_2\mu \in \mathrm{C}(K)'$ are invariant probability measures and thus $Q_1\mu = Q_2\mu$. This implies $Q_1 = Q_2$ and therefore $\EuScript{K}_{\mathrm{c}}(K;S)$ has a zero $Q$. If $\mu_1,\mu_2 \in \mathrm{C}(K)'$ are two probability measures, we also obtain $Q\mu_1 = Q\mu_2$. As a result $Q$ has rank one.\smallskip\\
Finally assume (iii). Let $Q \in \EuScript{K}_{\mathrm{c}}(K;S)$ be the zero which is a rank one operator. Take $x \in K$ and set $\mu \coloneqq Q\delta_x$. Since $Q$ {is} rank one, we obtain $Q\nu = Q\delta_x = \mu$ for each probability measure $\nu \in \mathrm{C}(K)'$. Now consider the operator $P \in \mathscr{L}(\mathrm{C}(K))$ given by
\begin{align*}
Pf \coloneqq \langle f,\mu\rangle \cdot \mathbbm{1}
\end{align*}
for $f \in \mathrm{C}(K)$. We then obtain
\begin{align*}
\langle Pf,\nu\rangle = \langle f,\mu\rangle \cdot  \langle \mathbbm{1},\nu\rangle = \langle f,\mu\rangle = \langle f, Q\nu \rangle
\end{align*}
for each $f \in \mathrm{C}(K)$ and each probability measure $\nu \in \mathrm{C}(K)'$. Hence $P' = Q$ and $Q$ is weak* continuous. Thus (i) and the remaining assertion follow from \cref{normweakly}.
\end{proof}
{
\begin{remark}
The equivalence of (i) and (ii) is also a direct consequence of Theorem 1.7 of \cite{Schr2013a}. The new part of \cref{uniquee} is the characterization of unique ergodicity via properties of the zero $Q \in \EuScript{K}_{\mathrm{c}}(K;S)$. In fact, we have proved the following for topological dynamical systems $(K;S)$ with $S$ left amenable.
\begin{enumerate}[(i)]
\item $(K;S)$ is weak* mean ergodic if and only if $\EuScript{K}_{\mathrm{c}}(K;S)$ has a zero (see \cref{mainfirstchap}).
\item $(K;S)$ is norm mean ergodic if and only if $\EuScript{K}_{\mathrm{c}}(K;S)$ has a weak* continuous zero (see \cref{normweakly}).
\item $(K;S)$ is uniquely ergodic if and only if $\EuScript{K}_{\mathrm{c}}(K;S)$ has a zero which is a rank one operator (see \cref{uniquee}).
\end{enumerate}
\end{remark}
}
{Recall that a topological dynamical system $(K;S)$ is \emph{minimal} if $K$ has no non-trivial closed $S$-invariant subsets.} The following consequence of \cref{uniquee} is a variation of \cite{KaWe1981}, Proposition 3.2{, for two-sided ergodic sequences} (see also the remark below Corollary 3.3 in \cite{Roma2011} and the paper by Iwanik \cite{Iwan1980}).
\begin{corollary}\label{minimalunique}
Consider a minimal topological dynamical system $(K;S)$ with $S$ left amenable. If there is an operator sequence $(T_n)_{n \in \N} \subset \mathrm{co}\,\EuScript{T}_S$ which is two{-}sided ergodic for $\EuScript{T}_S$ with respect to the $\sigma(\mathrm{C}(K),\ell^1(K))${-}topology such that $(T_nf)_{n \in \N}$ converges pointwise for each $f \in \mathrm{C}(K)$, then $(K;S)$ is uniquely ergodic.
\end{corollary}
\begin{proof}
Take a sequence $(T_n)_{n \in \N}$ as above and set $Pf(x)\coloneqq \lim_{n \rightarrow \infty} T_nf(x)$ for $x \in K$ and $f \in \mathrm{C}(K)$. Then $P$ maps $\mathrm{C}(K)$ to the space of Baire {$1$} functions $\mathrm{B}_1(K)$.\\
For $f\in \mathrm{C}(K)$ and $x_1,x_2 \in K$ the pre-images $M_i\coloneqq (Pf)^{-1}(Pf(x_i))$ are non{-}empty, $S${-}invariant {$\mathrm{G}_{\delta}$ sets} for $i=1,2$ and{---}by minimality of $(K;S)${---}dense. In Baire spaces the intersection of two dense {$\mathrm{G}_{\delta}$ sets} is dense and in particular non{-}empty. We conclude $Pf(x_1) = Pf(x_2)$ and therefore $Pf$ is constant. By Lebesgue's Theorem and \cref{uniquee} $(K;S)$ is uniquely ergodic.
\end{proof}
\begin{corollary}\label{uniquetame}
Every minimal tame metric topological dynamical system $(K;S)$ with $S$ amenable is uniquely ergodic.
\end{corollary}
\begin{proof}
By Proposition 1.3 of \cite{Schr2013a} there exists a two{-}sided {ergodic net} $(T_{\alpha})_{\alpha \in A}$ with respect to the weak topology $\sigma(\mathrm{C}(K),\mathrm{C}(K)')$. By passing to {a} subnet we may assume that the limit $Q\mu := \lim_{\alpha} T_{\alpha}'\mu$ exists in weak* topology for each $\mu \in \mathrm{C}(K)'$. Then $Q \in \mathrm{ker}\,\EuScript{K}_{\mathrm{c}}(K;S)${. Since $(K;S)$ is tame, $\EuScript{K}_{\mathrm{c}}(K;S)$ is a Rosenthal compact space and we }find a left {ergodic sequence} $(S_n')_{n \in \N} \subset \mathrm{co}\, \EuScript{T}_S'$ converging to $Q$. Since $QT_s'=T_s'Q = Q$ for each $s \in S$, the sequence $(S_n')_{n \in \N}$ is also two{-}sided ergodic. In particular $(S_n)_{n \in \N}$ is a two{-}sided {ergodic sequence} for $\EuScript{T}_S$ with respect to the $\sigma(\mathrm{C}(K),\ell^1(K))${-}topology such that $(S_nf)_{n \in \N}$ converges pointwise for each $f \in \mathrm{C}(K)$ and hence $(K;S)$ is uniquely ergodic by \cref{minimalunique}.
\end{proof}
\begin{remark}
\cref{uniquetame} has been proved for abelian group actions by E. Glasner ({see} Theorem 5.1 in \cite{Glas2007b}), D. Kerr and H. Li ({see} Theorem 7.19 \cite{KeLi2007}) as well as W. Huang  ({see} Theorem 4.8 in \cite{Huan2006}). Their proofs are based on a representation type result for minimal tame systems while {our proof uses} the topological {properties} of the semigroup $\EuScript{K}_{\mathrm{c}}(K;S)$.
\end{remark}
{The following theorem is the main result of this section.} {We recall that a topological dynamical system $(K;S)$ is \emph{topologically transitive} if there is $x \in K$ such that $\overline{Sx} =K$.}
\begin{theorem}\label{mainrelations}
Consider a topological dynamical system $(K;S)$ with $S$ left amenable and the following assertions.
\begin{enumerate}[(i)]
\item $(K;S)$ is weak* mean ergodic.
\item $(K;S)$ is norm mean ergodic.
\item $(K;S)$ is uniquely ergodic.
\end{enumerate}
Then (iii) $\Rightarrow$ (ii) $\Rightarrow$ (i). If $(K;S)$ is topologically transitive, {all these assertions are equivalent.}
\end{theorem}
\begin{remark}
Simple examples show that the three notions of {weak* mean ergodicity, norm mean ergodicity and unique ergodicity} are truly distinct.
\end{remark}
For the proof of \cref{mainrelations} we need two lemmas. We write $\mathrm{P}(K)$ for the probability measures on a compact space $K$ and remind the reader that we identify $K$ with the space of Dirac measures $\{\delta_x\mid x \in K\} \subset \mathrm{P}(K)$.
\begin{lemma}\label{orbitlem}
Consider a topological dynamical system $(K;S)$ with $S$ left amenable and a point $x \in K$. Then the following identities hold.
\begin{enumerate}[(i)]
\item $\EuScript{K}(K;S)(x) = \overline{Sx} \subset K.$
\item $\EuScript{K}_{\mathrm{c}}(K;S)(x) = \mathrm{P}(\overline{Sx}) \subset \mathrm{P}(K).$
\end{enumerate}
\end{lemma}
\begin{proof}
Since $\EuScript{K}(K;S)$ is compact, assertion (i) is obvious. For (ii) we obtain, by the Krein--Milman {t}heorem,
\begin{align*}
\mathrm{P}(\overline{Sx})  = \overline{\mathrm{co}\,\{\delta_y\mid y \in \overline{Sx}\}} = \overline{\mathrm{co}\,\{\delta_{sx}\mid s \in S\}} = \EuScript{K}_{\mathrm{c}}(K;S)(x),
\end{align*}
where the latter equation is a consequence of the compactness of $\EuScript{K}_{\mathrm{c}}(K;S)$.
\end{proof}
\begin{lemma}\label{weakstarunique}
  Consider a weak* mean ergodic topological dynamical system $(K;S)$ with $S$ left amenable. Then for each $x \in K$ the orbit system $(\overline{Sx};S)$ is uniquely ergodic.
\end{lemma}
\begin{proof}
Since subystems of weak* mean ergodic systems are again weak* mean ergodic, we may assume{---}by passing to an orbit system{---}that $(K;S)$ is transitive. Take $x \in K$ with $K = \overline{Sx}$ and let $Q \in \EuScript{K}_{\mathrm{c}}(K;S)$ be the zero element. The measure $\mu\coloneqq Qx \in \mathrm{C}(K)'$ is $S${-}invariant.\\
Now consider an invariant probability measure ${\nu} \in \mathrm{C}(K)'$. By \cref{orbitlem} we find an operator $T \in \EuScript{K}_{\mathrm{c}}(K;S)$ with ${\nu} = Tx$. This yields 
\begin{align*}
{\nu = Q\nu = QTx = Qx = \mu.}
\end{align*}
\end{proof}
\begin{proof}[of \cref{mainrelations}]
It is obvious that (iii) $\Rightarrow$ (ii) $\Rightarrow$ (i). If $(K;S)$ is transitive{,} then \cref{weakstarunique} proves the equivalence.
\end{proof}
{To conclude this section we characterize weak* and ``pointwise'' mean ergodicity for tame metric systems. Our theorem extends Theorem 4.5 of \cite{Roma2013} where a similar result was shown for $\N_0$-actions with metrizable Ellis semigroup (called \emph{ordinary systems}).}
\begin{theorem}\label{pointwisemain}
Consider a tame metric topological dynamical system $(K;S)$ with $S$ amenable. Then the following assertions are equivalent.
\begin{enumerate}[(i)]
\item For each operator net $(T_{\alpha})_{\alpha \in A}$ which is right ergodic for $\EuScript{T}_S$ with respect to the $\sigma(\mathrm{C}(K),\ell^1(K))${-}topology, each $f \in \mathrm{C}(K)$ and each $x \in K$ the limit $\lim_{\alpha} (T_\alpha f(x))_{\alpha \in A}$ exists.
\item The system $(K;S)$ is weak* mean ergodic.
\item For each $x \in X$ the system $(\overline{Sx};S)$ contains a unique minimal set.
\end{enumerate}
\end{theorem}
Once again we need two lemmas. The first one is a generalization of Lemma 2.3 of \cite{Roma2011} to our setting.
\begin{lemma}\label{idempotent}
Let $(K;S)$ be a topological dynamical system and consider $\psi \in \mathrm{ker}\,\mathrm{E}(K;S)$. Then $\psi(K)$ is contained in the union of minimal sets.
\end{lemma}
\begin{proof}
By Theorem 1.2.12 in \cite{BeJuMi1989} we find a minimal left ideal $I$ of $\mathrm{E}(K;S)$ containing $\psi$. However, the set $I(x)$ is minimal by Proposition 1.6.12 in \cite{BeJuMi1989}.
\end{proof}
We need a more general version of \cref{uniquetame}.
\begin{lemma}\label{tameunique2}
Let $(K;S)$ be a tame metric topological dynamical system with $S$ amenable containing a unique minimal subset. Then $(K;S)$ is uniquely ergodic. 
\end{lemma}
\begin{proof}
Denote the unique minimal subset by $M$ and consider an invariant probability measure $\mu \in \mathrm{C}(K)'$. Since the support of $\mu$ is closed and invariant, it contains $M$.\\
Now take a minimal idempotent $\psi = \lim_{n \rightarrow \infty} \color{red} \varphi_{s_n} \color{black} \in \mathrm{E}(K;S)$. Then $\psi(K) \subset M$ by \cref{idempotent} and for each positive $f \in \mathrm{C}(K)$ vanishing on $M$ we obtain
\begin{align*}
\langle f,\mu \rangle = {\lim_{n \rightarrow \infty}} \langle f, T_{s_n}'\mu \rangle  = \langle f, T_{\psi}'\mu \rangle = \langle T_{\psi}f, \mu \rangle =0.
\end{align*}
This shows $\supp \mu = M$ and \cref{uniquetame} proves the claim.
\end{proof}
Before proceeding to the proof of \cref{pointwisemain}, we observe that pointwise and weak ergodic nets are the same for tame systems.
\begin{lemma}\label{tameergodicnets}
Consider a tame metric topological dynamical $(K;S)$ with $S$ amenable. Each operator net $(T_{\alpha})_{\alpha \in A}$ which is right ergodic for $\EuScript{T}_S$ with respect to the $\sigma(\mathrm{C}(K),\ell^1(K))${-}topology is right ergodic for $\EuScript{T}_S$ with respect to the $\sigma(\mathrm{C}(K),\mathrm{C}(K)')${-}topology.
\end{lemma}
\begin{proof}
Take an operator net $(T_{\alpha})_{\alpha \in A}$ which is right ergodic with respect to the $\sigma(\mathrm{C}(K),\ell^1(K))${-}topology, $f \in \mathrm{C}(K)$ and $s \in S$. We then have $\lim_{\alpha} (T_{\alpha}(f-T_sf))(x) = 0$ for all $x \in K$.\\
Equip $\mathrm{B}_1(K)$ with the topology of pointwise convergence. Since $(K;S)$ is tame, the set $\overline{\mathrm{co}\,\EuScript{T}_S (\mathrm{Id}-T_s)f}$ is compact in $\mathrm{B_1}(K)$ and it contains the net $(T_{\alpha}(\mathrm{Id}-T_s)f)_{\alpha \in A}$. The main theorem of \cite{Rose1977} therefore implies 
\begin{align*}
\lim_{\alpha} T_{\alpha}(\mathrm{Id}-T_s)f = 0
\end{align*}
with respect to the $\sigma(\mathrm{C}(K),\mathrm{C}(K)')${-}topology.
\end{proof}
\begin{proof}[of \cref{pointwisemain}]
If $(K;S)$ is weak* mean ergodic, then each orbit is uniquely ergodic by \cref{weakstarunique} and{---}since every minimal set supports an invariant probability measure{---}contains only one minimal set. {This proves the implication \enquote{(ii) $\Rightarrow$ (iii)}.} \smallskip\\
By \cref{tameunique2} assertion (iii) implies that each orbit is uniquely ergodic. Thus, for each right {$\sigma(\mathrm{C}(K),\mathrm{C}(K)')$-ergodic net} $(T_{\alpha})_{\alpha \in A}$ for $\EuScript{T}_S$ we obtain that $(T_{\alpha}f)_{\alpha \in A}$ converges weakly and thus pointwise on each orbit. But then $(T_{\alpha}f)_{\alpha \in A}$ converges pointwise on $K$. Combined with \cref{tameergodicnets} this implies (i).\smallskip\\
Finally suppose that (i) holds and take two elements $Q_1,Q_2 \in \ker(\EuScript{K}_{\mathrm{c}}(K;S))$. Since $\EuScript{K}_{\mathrm{c}}(K;S)$ is a Fréchet{--Urysohn} space, we find {right ergodic sequences} \\
$(T_{i,n}')_{n \in \N}\subset \mathrm{co}\, \EuScript{T}_S'$ for $\EuScript{T}_S'$ converging to $Q_i$ for $i=1,2$. Consider the sequences $(T_{i,n})_{n \in \N}$ consisting of the pre{-}adjoints and observe that for each $f \in \mathrm{C}(K)$ we have $\lim_{n \rightarrow {\infty}} T_{i,n}f = Q_i'f$ with respect to the $\sigma(\mathrm{C}(K)'',\mathrm{C}(K)')${-}topology for $i=1,2$. By assumption the sequence obtained by alternating the members of $(T_{1,n}f)_{n \in \N}$ and $(T_{2,n}f)_{n \in \N}$ converges pointwise and thus, by Lebegue's Theorem, in {the} $\sigma(\mathrm{C}(K)'',\mathrm{C}(K)')${-}topology. This yields $Q_1 = Q_2$.
\end{proof}
The next example shows that even for tame systems the weak* convergence of a single ergodic sequence does not ensure weak* mean ergodicity.
\begin{example}\label{rolandex}
Consider the space $\{0,1\}^{\N}$ with the product topology (which is compact and metrizable) and endow it with the shift $\varphi$ given by \\
$\varphi \left((a_n)_{n \in \N}\right) \coloneqq (a_{n+1})_{n \in \N}$ for $(a_n)_{n \in \N} \in \{0,1\}^{\N}$. 
Consider the point $x = (x_n)_{n \in \N} \in  \{0,1\}^{\N}$ with
\begin{align*}
x_n = \begin{cases} 1 & \text{if } n \in \{k(N)+1,...,k(N)+N\},\\
0 & \text{if } n \in \{k(N)+N+1,...,k(N+1)\},
\end{cases}
\end{align*}
where
\begin{align*}
k(N)\coloneqq \sum_{n=1}^{N-1} (n+10^n) = \frac{N \cdot (N-1)}{2} + 10 \cdot \frac{10^{N-1}-1}{9}
\end{align*}
for $N \in \N$. 
For illustration we give the start of this sequence as
\begin{align*}
x = (1,\underbrace{0,......,0}_{10 \text{ zeroes}},1,1,\underbrace{0,......,0}_{100 \text{ zeroes}},1,1,1,\underbrace{0,......,0}_{1000 \text{ zeroes}},1,1,1,1,0,...........).
\end{align*}
Now consider the compact subspace $K\coloneqq \overline{\{\varphi^n(x)\mid n \in \N_0\}}$ and the system $(K;S)$ with $S\coloneqq \{(\varphi|_K)^n\mid n \in \N_0\}$. {It is easy to see that} $K$ is countable and thus $\mathrm{E}(K;S) \subset K^K$ has cardinality at most $\mathfrak{c}$. Therefore the system is tame by Theorem 1.2 of \cite{Glas2006a}.\\
For each $f \in \mathrm{C}(K)$ the Cesàro means $(\frac{1}{N}\sum_{n=0}^{N-1}T_{\varphi}^nf)_{n\in \N}$ converge pointwise (and therefore with respect to the weak* topology) to the function $Pf\colon K \longrightarrow \C$ with
\begin{align*}
Pf((x_n)_{n \in \N}) \coloneqq \begin{cases} f((1)_{n \in \N})&  \text{if there is } N \in \N \text{ with } x_n = 1 \text{ for all } n \geq N,\\
f((0)_{n \in \N}) & \text{else}.
\end{cases}
\end{align*}
On the other hand, the constant zero sequence and the constant one sequence are two fixed points of the system. Thus $(K;S)$ is not weak* mean ergodic by \cref{pointwisemain}.
\end{example}

\parindent 0pt
\parskip 0.5\baselineskip
\setlength{\footskip}{4ex}
\bibliographystyle{alpha}
\bibliography{literature} 
\footnotesize

  \textsc{Henrik Kreidler, Mathematisches Institut, Universität Tübingen, Auf der Morgenstelle 10, D-72076 Tübingen, Germany}\par\nopagebreak
  \textit{E-mail address}: \texttt{hekr@fa.uni-tuebingen.de}
\end{document}